\documentclass[10pt]{amsart}

\newtheorem{lem}{Lemma}[section]
\newtheorem{thm}{Theorem}[section]
\newtheorem{prop}{Proposition}[section]
\newtheorem{cor}{Corollary}[section]
\newtheorem{obs}{Remark}[section]
\newtheorem{defin}{Definition}[section]

\usepackage{amssymb}
\usepackage{amsmath}
\usepackage{amsfonts}
\usepackage{graphicx}

\numberwithin{equation}{section}

\def\k0{\kappa_0}

\def\lgl{\langle}
\def\rgl{\rangle}
\def\bfe{{\mathsf{e}}}
\def\bfE{{\mathsf{E}}}

\def\bfu{{\bf{u}}}

\def\bfx{{\bf{x}}}
\def\bfo{{\bf{0}}}

\def\mR3{{\mathbb{R}^3}}

\begin{document}
\title[Energy Cascades]
{Energy cascades and flux locality in physical scales of the 3D Navier-Stokes equations}
\author{R. Dascaliuc}
\address{Department of Mathematics\\
University of Virginia\\ Charlottesville, VA 22904}
\author{Z. Gruji\'c}
\address{Department of Mathematics\\
University of Virginia\\ Charlottesville, VA 22904}
\date{\today}
\begin{abstract}
Rigorous estimates for the total -- (kinetic) energy plus pressure
-- flux in $\mathbb{R}^3$ are obtained from the three dimensional
Navier-Stokes equations. The bounds are used to establish a
condition -- involving Taylor length scale and the size of the
domain -- sufficient for existence of the inertial range and the
energy cascade in decaying turbulence (zero driving force,
non-increasing global energy). Several manifestations of the
locality of the flux under this condition are obtained. All the
scales involved are {\em actual physical scales} in $\mathbb{R}^3$
and no regularity or homogeneity/scaling assumptions are made.
\end{abstract}
\maketitle

\section{introduction}

One of the main features of Kolmogorov's empirical turbulence theory
\cite{Kol1,Kol2,Kol3} is existence of {\em energy cascade} over a
wide range of length scales, called the {\em inertial range}, where the
dissipation effects are dominated by the transport of energy from
higher to lower scales. Energy cascades have been observed in
physical experiments, but theoretical justification of this
phenomenon using equations of fluid motion, and in particular, the
Navier-Stokes equations (NSE), remains far from being settled.
The technical complexity of the NSE makes it difficult to establish 
conditions under which such cascades can occur.  A particular
problem is the possible lack of regularity of the solutions to
the NSE, and thus choosing the right setting becomes crucial. (For
an overview of various mathematical models of turbulence and the
theory of the NSE, see, e.g., \cite{FMRTbook,Fbook,ES} and
\cite{L-R,CFbook,Tbook1}, respectively.)

\medskip

The first studies in this direction were made in \cite{FMRT}, where
infinite-time averages of the Leray-Hopf solutions in the Fourier
setting were used to establish a sufficient condition for the energy
cascade. This condition, involving Taylor length scale, provided an
inspiration for the sufficient condition (\ref{scales_con_fin})
obtained in section 4. In contrast to \cite{FMRT}, our goal was to
work in {\em physical space}, dealing with actual length scales in
$\mR3$ rather than the Fourier wave numbers.

\medskip

In studying a PDE model, a natural way of introducing a
concept of scale is to measure oscillations, i.e., (distributional)
derivatives of a quantity with respect to the scale.

Considering an $L^1_{loc}$ function $f$ on a ball of radius $2R$,
$B(\bfx_0, 2R)$, the \emph{physical scale $R$} is introduced via
bounds on the distributional derivatives of $f$ where a test
function $\psi$ is a refined -- smooth, non-negative, equal to 1 on
$B(\bfx_0, R)$ and featuring optimal bounds on the derivatives over
the outer $R$-layer -- cut-off function on $B(\bfx_0, 2R)$.
(Uniformity in all scales dictates linearity of the length of the
outer layer in $R$; hence $B(\bfx_0, R+R)$.)

More explicitly,

\[
|(D^\alpha f, \psi)| \le \int_{B(\bfx_0, 2R)} |f| |D^\alpha \psi|
\le \Bigl(c(\alpha) \frac{1}{R^{|\alpha|}} |f| ,
\psi^{\rho(\alpha)}\Bigr)
\]
for some $c(\alpha)>0$ and $\rho(\alpha)$ in $(0,1)$.

(An attempt to introduce a concept of scale via characteristic
functions in place of smooth cut-off functions would
lead to infinite concentration -- delta
functions -- invalidating much of the desired calculus.)

This approach has a similar flavor as introducing the \emph{Fourier scale}
$|\xi|$ via
\[
 \widehat{D^\alpha f} (\xi) = i^{|\alpha|} \xi^\alpha \hat{f}(\xi)
\]
(in the Schwarz space, and then by duality in the space of tempered
distributions).

\medskip

Let $\bfx_0$ be in $B(\bfo,R_0)$ ($R_0$ being the integral scale,
$B(\bfo,2R_0) \subset \Omega$ where $\Omega$ is the global spatial
domain) and $0< R \le R_0$. Define local -- per unit of mass -- (kinetic) energy, $\bfe$ and
enstrophy, $\bfE$, at time $t$, associated with the ball $B(\bfx_0,R)$
by
\[
\bfe_{\bfx_0,R}(t)=\int \frac{1}{2}|\bfu|^2\phi^{2\delta-1} \,d\bfx\;,
\]
\[
\bfE_{\bfx_0,R}(t)=\int |\nabla\otimes \bfu|^2\phi \, d\bfx\;,
\]
where $\phi=\eta \, \psi$ and $\eta$ and $\psi$ are refined cut-off
functions in time and space, respectively (for some $\frac{1}{2} <
\delta < 1$).

A total flux -- (kinetic) energy plus pressure -- through the boundary of
a region $D$ is given by
\[
\int\limits_{\partial D}\biggl(\frac{1}{2}|\bfu|^2+
p\biggr)\bfu\cdot \mathbf{n} \,ds  \, = \,
\int\limits_{D}\bigl[(\bfu \cdot \nabla)\bfu + \nabla p\bigr] \cdot
\bfu \, d\bfx
\]
where $\mathbf{n}$ is an outward normal. Considering the NSE
localized to $B(\bfx_0,2R)$ -- and utilizing $\, \nabla\cdot \bfu
= 0$ -- leads to a localized flux,
\[
\Phi_{\bfx_0,R}(t)=\int \biggl(\frac{1}{2}|\bfu|^2+p\biggr) \bfu
\cdot \nabla \phi \, d\bfx = - \int \bigl[(\bfu \cdot \nabla)\bfu +
\nabla p\bigr] \cdot \bfu \, \phi \, d\bfx.
\]
Since $\psi$ can be constructed such that $\nabla \phi=\eta \,
\nabla \psi$ is oriented along the radial directions of $B(\bfx_0,
2R)$ \emph{toward the center of the ball}, $\Phi_{\bfx_0,R}$
represents the flux {\em into} $B(\bfx_0,R)$ through the layer
between the spheres $S(\bfx_0,2R)$ and $S(\bfx_0,R)$ ($\nabla \phi
\equiv 0$ on $B(\bfx_0, R)$).

A more dynamic physical significance of the sign of $\Phi_{\bfx_0,
R}$ can be seen from the equations: multiplying the NSE by $\psi
\bfu$ and integrating over $B(\bfx_0, 2R)$ (formally, assuming
smoothness) leads to
\begin{equation}\label{loc_tr}
\frac{d}{dt} \int \frac{1}{2} |\bfu|^2 \psi \, d\bfx = \Phi_{\bfx_0,
R} + \nu \int \triangle \bfu \cdot \bfu \, \psi \, d\bfx.
\end{equation}
Plainly, the positivity of $\Phi_{\bfx_0, R}$ contributes to the
increase of the kinetic energy around the point $\bfx_0$ at scale
$R$.

Since the flux consists of both the kinetic and the pressure parts,
a natural question is whether there is a transfer of the kinetic
energy from larger scales into $B(\bfx_0,R)$, or perhaps the
increase is mainly due to the change in pressure. In general, it is
possible that the increase of the kinetic energy around $\bfx_0$ is
due solely to the pressure part; a simple example being
$\bfu=(c,c,c) \, t, \, p=cx_1+cx_2+cx_3$. However, \emph{in physical
situations where the kinetic energy on the (global) spatial domain
$\Omega$ is non-increasing}, e.g., a bounded domain with no-slip
boundary conditions, or the whole space with either decay at
infinity or periodic boundary conditions (here, we are concerned
with the case of decaying turbulence, setting the driving force to
zero), \emph{the increase of the kinetic energy in $B(\bfx_0,R)$ --
and consequently, the positivity of $\Phi_{\bfx_0,R}$ -- implies
local transfer of the kinetic energy from larger scales} simply
because the local kinetic energy is increasing while the global
kinetic energy is non-increasing resulting in decrease of the
kinetic energy in the complement. This is also consistent with the
fact that in the aforementioned scenarios one can project the NSE --
in an appropriate functional space -- to the subspace of
divergence-free functions effectively eliminating the pressure and
revealing that the local flux $\Phi_{\bfx_0,R}$ is indeed driven by
transport/inertial effects rather than the change in the pressure.
(The $\bfu=(c,c,c) \, t, \, p=cx_1 +cx_2+cx_3$ example pertains to a
completely opposite situation, the kinetic energy is simply
uniformly growing over the whole spatial domain.)

Henceforth, following the discussion in the preceding paragraphs --
in the setting of decaying turbulence (zero driving force,
non-increasing global energy) -- the positivity and the negativity
of $\Phi_{\bfx_0, R}$ will be interpreted as a transfer of (kinetic)
energy around the point $\bfx_0$ at scale $R$ toward smaller scales
and a transfer of (kinetic) energy around the point $\bfx_0$ at scale
$R$ toward larger scales, respectively.

Completely analogous definitions hold for shells of radii $R$ and
$2R$.

We also consider finite time averages of each of the aforementioned quantities.

\medskip

Our goal is to obtain a manifestation of the (kinetic) energy
cascade in \emph{physical space}, i.e., formulate a condition on
$B(\bfo,R_0)$ that would imply that the time-averaged energy
transfers/cascades to smaller scales across a range of scales (the
existence of the inertial range).

A key point here is that we do not assume \emph{any homogeneity} of the
flow; hence one can not expect to show that the local fluxes are
positive for each individual ball $B(\bfx, R)$. The best one can
hope for is to prove the positivity of the flux over \emph{some
spatial average}.

We choose to work with a very straightforward spatial average: the
arithmetic mean of the local fluxes -- time-averaged, per unit mass
-- computed over a family of coverings of $B(\bfo,R_0)$, the
so-called optimal coverings.

Let $K_1$ and $K_2$ be two positive integers. A covering
$\{B(\bfx_i,R)\}_{i=1}^n$ of $B(\bfo,R_0)$ is an \emph{optimal
covering} (with parameters $K_1$ and $K_2$) if
\[
 \biggl(\frac{R_0}{R}\biggr)^3 \le n \le K_1
 \biggr(\frac{R_0}{R}\biggr)^3,
\]
and any point $\bfx$ in $B(\bfo,R_0)$ is covered by at most $K_2$
balls $B(\bfx_i,2R)$. (Optimal coverings exist for all large enough
$K_1$ and $K_2$, the critical values depending only on dimension of
the space. In $\mathbb{R}^3$, we can take $K_1=K_2=8$.)

Let $f$ be a sign-varying quantity (e.g., the flux density $ -
[(\bfu \cdot \nabla)\bfu+\nabla p] \cdot \bfu$), and consider the
arithmetic mean of the quantity locally averaged over the (optimal)
covering elements $B(\bfx_i, R)$,
\[
 F_R = \frac{1}{n} \sum_{i=1}^n \frac{1}{R^3} \int_{B(\bfx_i, 2R)} f \,
 \psi_i^\rho \, d\bfx
\]
(for some $0 < \rho \le 1$).

A revealing observation is that $F_{R} \sim \, \mbox{const} \, (R) \
$ for \emph{all} optimal coverings at scale $R$ ($K_1$ and $K_2$ fixed)
indicates there are no significant fluctuations of sign of $f$ at
scales comparable or greater than $R$. In other words, if there are
significant fluctuations of sign of $f$ at scale $R^*$, $F_R$ will
run over a wide range of values while the average is being run over
all permissible optimal coverings (determined by  $K_1$ and $K_2$), for
any $R$ comparable or less than $R^*$.

When there is no change of sign at all, i.e., in the case of a
signed quantity (e.g., the energy density $f=\frac{1}{2}|\bfu|^2, \,
\rho = 2\delta-1$ or the enstrophy density $f=|\nabla \otimes
\bfu|^2, \, \rho=1$), one would then expect that for any scale $R$,
$0 < R \le R_0$, the averages $F_R$ are all comparable to each
other. This is in fact true (an easy proof).

Utilizing the NSE \emph{via the local energy inequality} -- in the
mathematical setting of suitable weak solutions \cite{S, CKN} -- we
establish the positivity and near-constancy (comparable to $\nu \bfE$,
where $\nu$ is the viscosity and $\bfE$ is the average enstrophy over
$B(\bfo,2R_0)\times (0,2T)$) of the averaged flux across a range of
scales under a very simple and natural (in the sense of turbulence
phenomenology) condition; namely, that Taylor micro-scale $\tau_0$
associated with $B(\bfo, R_0)$ is smaller than the integral spatial
scale $R_0$ (cf. (\ref{scales_con_fin})). The larger the gap, the
deeper the inertial range. This condition is reminiscent of the
Poincar\'e inequality on a domain of the corresponding size (see
Remark \ref{Rem_4.2}); moreover, the condition in hand would be easy
to check in physical experiments as the averages involved are very
straightforward. In addition, the length of the time interval $T$ is
consistent with the intrinsic scaling of the model (cf.
(\ref{T_con})).

It is interesting to interpret the cascade in the light of the above observation 
regarding the meaning of near-constancy of optimal cover averages. Essentially,
for any $\tau_0 \le R \le R_0$ (within the inertial range), the flux
density does not experience significant fluctuations of sign at scale
$R$; the significant fluctuations of sign are only possible at the scales
substantially smaller than $\tau_0$, i.e., inside the dissipation
range.

\medskip

The second part of the paper concerns {\em locality} of the flux. It
is believed (see \cite{O}) that the energy flux inside the inertial
range of turbulent flows depends strongly on the flow in nearby
scales, its dependence on the lower and much higher scales being
weak. The theoretical proof of this conjecture remained elusive. The
first quantitative results on fluxes were obtained by early 70's
(see \cite{Kr}). Much later, the authors in \cite{LF} used the NSE
in the Fourier setting to explore locality of scale interactions for
statistical averages, while the investigation in \cite{E} revealed
the locality of filtered energy flux under an assumption that
solutions to the vanishing viscosity Euler's equations saturate a
defining inequality of a suitable Besov space, i.e., under a (weak)
scaling assumption. A more recent work \cite{CCFS} provided a proof
of the locality of the energy flux in the setting of the
Littlewood-Paley decomposition.

In the last section we prove the locality of the energy cascade --
in decaying turbulence -- in the physical space throughout the
inertial range established in Theorem 4.1. In particular,
considering dyadic shells at the scales $2^k R$ ($k$ an integer) in
the physical space, we show that both ultraviolet and infrared
locality propagate \emph{exponentially} in the shell number $k$.

To the best of our knowledge, the condition (\ref{scales_con_fin})
is presently the only condition (in any solution setting) implying
both the existence of the inertial range and the locality of the
energy flux. Moreover, it does not involve any additional regularity
or homogeneity/scaling assumptions on the solutions to the NSE.

\section{preliminaries}
We consider three dimensional incompressible Navier-Stokes equations (NSE)
\begin{equation}\label{inc-nse}
\begin{aligned}
\frac{\partial}{\partial t}\bfu(t,\bfx)-\nu\Delta \bfu(t,\bfx)
+(\bfu(t,\bfx)\cdot\nabla)\bfu(t,\bfx)+\nabla p(t,\bfx)&=0,\\
\nabla\cdot\bfu(t,\bfx)&=0\;
\end{aligned}
\end{equation}
where the space variable $\bfx$ is in $\mathbb{R}^3$ and the time
variable $t$ is in $(0, \infty)$. The vector-valued function $\bfu$
and the scalar-valued function $p$ represent the fluid velocity and
the pressure, respectively, while the constant $\nu$ is the
viscosity of the fluid.

Since our goal is to investigate local fluxes in the physical space,
the class of suitable weak solutions (see \cite{CKN,L-R}) will
provide an appropriate mathematical framework.

\begin{defin} Let $\Omega$ be an open connected set in $\mR3$. We say that
$(\bfu, p)$ is a \emph{suitable weak solution} on $(0, \infty)
\times \Omega$ if
\begin{enumerate}
\item[(a)] $\bfu\in L^{\infty}((0,\infty),L^2(\Omega)^3)\cap L^{2}
((0,\infty),H^1(\Omega)^3)$ and $p\in L^{3/2}((0,\infty)\times\Omega)$;
\item[(b)] the NSE (\ref{inc-nse}) are satisfied in the weak (distributional) sense;
\item[(c)] the local energy inequality is satisfied: for any
$\phi\in\mathcal{D}((0,\infty)\times\Omega)$, $\phi\ge0$ we have
\begin{equation}\label{loc_ene_ineq}
2\nu\iint|\nabla\otimes\bfu|^2\phi\,d\bfx\,dt \le
\iint|\bfu|^2(\partial_t\phi+\nu\Delta\phi)\,d\bfx\,dt
+\iint(|\bfu|^2+2p)\bfu\cdot\nabla\phi\,d\bfx\,dt\;
\end{equation}
\end{enumerate}
\end{defin}
\noindent where $\mathcal{D}((0,\infty)\times\Omega)$ denotes the
space of infinitely differentiable functions with compact support in
$(0, \infty) \times \Omega$.

The existence of the suitable weak solutions in the case where
$\Omega = \mathbb{R}^3$ and the external force is zero, given a
divergence-free initial condition in $L^2$, was first established in
\cite{S}. See also \cite{CKN,L-R} for more general results related
to existence and regularity properties of the suitable weak
solutions.

A solution to the NSE on $(0, \infty) \times \Omega$ is called {\em
regular} if its $H^1$ norm is bounded on $(0,T)$ for any $T$
positive. Given appropriate boundary conditions, this implies that
the solution is infinitely differentiable (in fact, analytic) in
both space and time and so it is a classical physical solution. In
particular, the local energy \emph{equality} holds
((\ref{loc_ene_ineq}) becomes an equality). The smoothness of the
suitable weak solutions to the NSE is still an open problem, and the
best result in this direction reads that the one-dimensional
(parabolic) Hausdorff measure of the singular set in $(0,T) \times
\Omega$ is zero \cite{CKN} (outside the singular set, a suitable
weak solution is infinitely differentiable in the spatial
variables).

%

In what follows, we consider
\begin{equation}\label{omega_ass}
R_0>0\quad\mbox{such that}\quad B(\bfo,3R_0)\subset\Omega\;
\end{equation}
where $B(\bfo,3R_0)$ denotes the ball in $\mR3$ centered at the
origin and with the radius $3R_0$.

Let $1/2\le\delta<1$. Choose $\psi_0\in\mathcal{D}(B(\bfo,2R_0))$
satisfying
\begin{equation}\label{psi0}
0\le\psi_0\le 1,\quad\psi_0=1\ \mbox{on}\ B(\bfo,R_0),
\quad\frac{|\nabla\psi_0|}{\psi_0^{\delta}}\le\frac{C_0}{R_0}, \quad
\frac{|\triangle\psi_0|}{\psi_0^{2\delta-1}}\le\frac{C_0}{R_0^2}\;.
\end{equation}
 For a $T>0$ (to be chosen later), $\bfx_0\in B(\bfo,R_0)$ and $0<R\le R_0$ define
$\phi=\phi_{\bfx_0,T,R}(t,\bfx)=\eta(t)\psi(\bfx)$ to be used in the
local energy inequality (\ref{loc_ene_ineq}) where $\eta=\eta_T(t)$
and $\psi=\psi_{\bfx_0,R}(\bfx)$ are refined cut-off functions
satisfying the following conditions,
\begin{equation}\label{eta_def}
\eta\in\mathcal{D}(0,2T),\quad 0\le\eta\le1,\quad\eta=1\ \mbox{on}\
(T/4,5T/4),\quad\frac{|\eta'|}{\eta^{\delta}}\le\frac{C_0}{T}\; ;
\end{equation}
if $B(\bfx_0,R)\subset B(\bfo,R_0)$, then
$\psi\in\mathcal{D}(B(\bfx_0,2R))$ with
\begin{equation}\label{psi_def}\begin{aligned}
\ 0\le\psi\le\psi_0,\quad\psi=1\ \mbox{on}\ B(\bfx_0,R)\cap
B(\bfo,R_0),
\quad\frac{|\nabla\psi|}{\psi^{\delta}}\le\frac{C_0}{R}, \
\frac{|\triangle\psi|}{\psi^{2\delta-1}}\le\frac{C_0}{R^2}\;,
\end{aligned}
\end{equation}
and if  $B(\bfx_0,R)\not\subset B(\bfo,R_0)$, then
$\psi\in\mathcal{D}(B(\bfo,2R_0))$ with $\psi=1\ \mbox{on}\ B(\bfx_0,R)
\cap B(\bfo,R_0)$ satisfying, in addition to (\ref{psi_def}), the
following:
\begin{equation}\label{psi_def_add1}
\begin{aligned}
&
\psi=\psi_0\ \mbox{on the part of the cone in}\ \mR3\ \mbox{centered at zero and passing through}\\
& S(\bfo,R_0)\cap B(\bfx_0,R)\ \mbox{between}\  S(\bfo,R_0)\
\mbox{and}\ S(\bfo,2R_0)
\end{aligned}
\end{equation}
and
\begin{equation}\label{psi_def_add2}
\begin{aligned}
&
\psi=0\ \mbox{on}\ B(\bfo,R_0)\setminus B(\bfx_0,2R)\ \mbox{and outside the part of the cone in}\ \mR3\\
 &
 \mbox{centered at zero and passing through}\ S(\bfo,R_0)\cap B(\bfx_0,2R)\\
 &
 \mbox{between}\  S(\bfo,R_0)\ \mbox{and}\ S(\bfo,2R_0).
\end{aligned}
\end{equation}

Figure \ref{ball_fig} illustrates the definition of $\psi$ in the case $B(\bfx_0,R)$ is not
 entirely contained in
$B(\bfo,R_0)$.

\begin{figure}
  \centerline{\includegraphics[scale=1, viewport=188 449 493 669, clip] {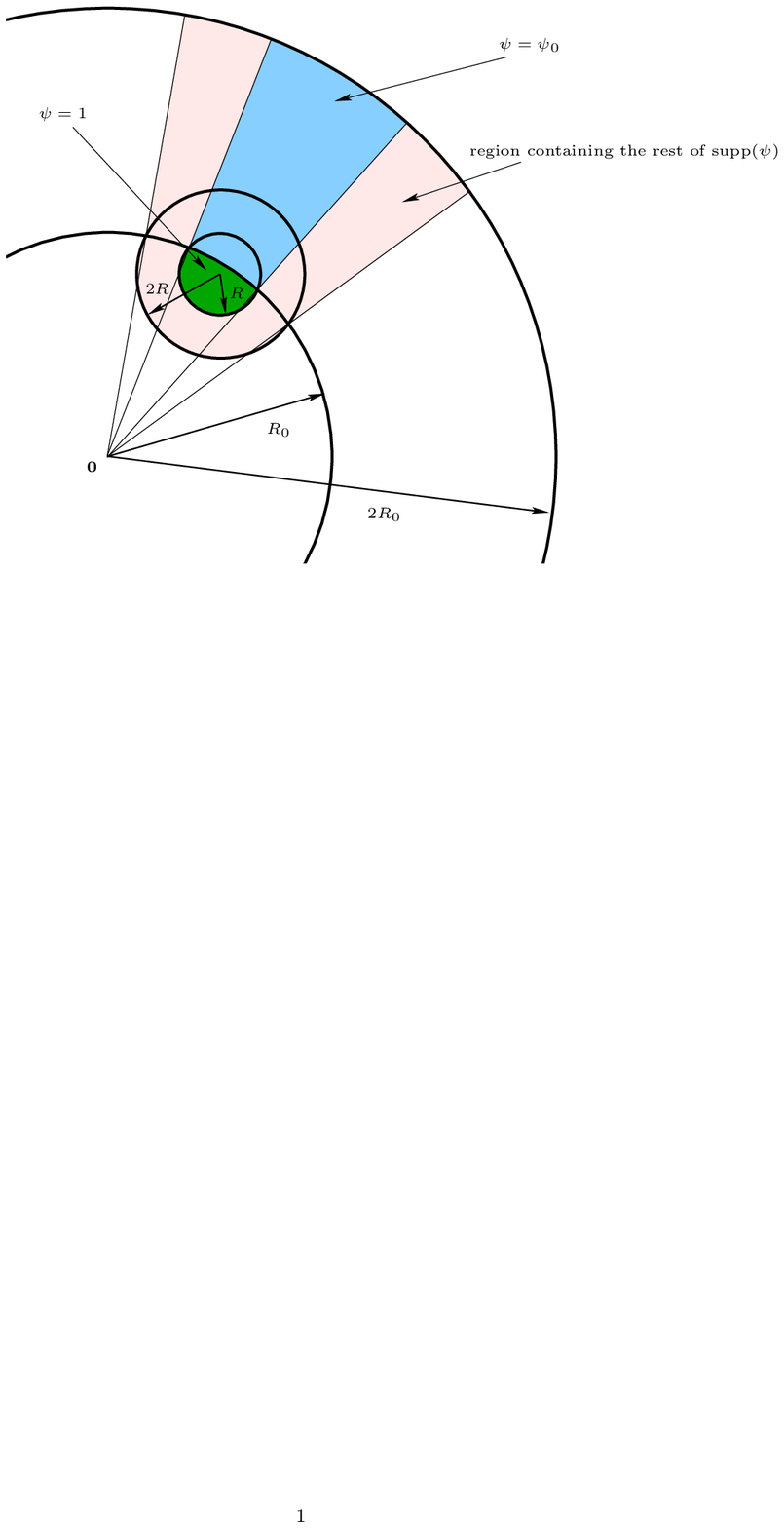}}
  \caption{Regions of supp$(\psi)$ in the case $B(\bfx_0,R)\not\subset B(\bfo,R_0)$,
  cross-section.}
  \label{ball_fig}
\end{figure}

\begin{obs}{\em
The additional conditions on the boundary elements
(\ref{psi_def_add1}) and (\ref{psi_def_add2}) are necessary to
obtain the lower bound on the fluxes in terms of the same version of
the localized enstrophy $E$ in Theorems \ref{balls_thm} and
\ref{shells_thm} (see Remarks \ref{E'_rem1} and \ref{E'_rem3}).
}\end{obs}

\section{Localized Energy, Enstrophy and Flux; Ensemble Averages}

Let $\bfx_0\in B(\bfo,R_0)$ and $0<R\le R_0$. Define localized
energy, $\bfe$, and enstrophy, $\bfE$, at time $t$ -- all per unit of mass -- associated with
$B(\bfx_0,R)$ by
\begin{equation}\label{enerdef}
\bfe_{\bfx_0,R}(t)=\int \frac{1}{2}|\bfu|^2\phi^{2\delta-1}\,d\bfx\;,
\end{equation}
\begin{equation}\label{enstdef}
\bfE_{\bfx_0,R}(t)=\int |\nabla\otimes\bfu|^2\phi\,d\bfx\;
\end{equation}
(for some $\frac{1}{2} < \delta < 1$).

The total -- (kinetic) energy plus pressure -- flux through sphere
$S(\bfx_0,R)$ is given by
\[\int\limits_{S(\bfx_0,R)}(\frac{1}{2}|\bfu|^2+ p)\bfu\cdot{\bf{n}}\,ds=
\int\limits_{B(\bfx_0,R)}[(\bfu\cdot\nabla)\bfu + \nabla p]
\cdot\bfu\,d\bfx\;
\] where ${\bf{n}}$ is an outward normal. Considering the NSE localized
to $B(\bfx_0,R)$ leads to a localized version of the flux,
\begin{equation}\label{fluxdef}
\Phi_{\bfx_0,R}(t)=\int
(\frac{1}{2}|\bfu|^2+p)\bfu\cdot\nabla\phi\,d\bfx = - \int \bigl[(\bfu
\cdot \nabla)\bfu + \nabla p\bigr] \cdot \bfu \, \phi \, d\bfx
\end{equation}
where $\phi=\eta\psi$ with $\eta$ and $\psi$ as in
(\ref{eta_def}-\ref{psi_def}). Since $\psi$ can be constructed such
that $\nabla\phi=\eta\nabla \psi$ is oriented along the radial
directions of $B(\bfx_0,R)$ towards the center of the ball $\bfx_0$,
$\Phi(\bfx_0,R)$ represents the flux {\em into} $B(\bfx_0,R)$
through the layer between the spheres $S(\bfx_0,2R)$ and
$S(\bfx_0,R)$ (in the case of the boundary elements satisfying the
additional hypotheses (\ref{psi_def_add1}) and (\ref{psi_def_add2}),
$\psi$ is almost radial and the gradient still points inward).

For a quantity $\Theta_{\bfx,R}(t)$, $t\in[0,2T]$ and a covering
$\{B(\bfx_i,R)\}_{i=1,n}$ of $B(\bfo,R_0)$ define a \emph{time-space
ensemble average}
\begin{equation}
\lgl\Theta\rgl_R=\frac{1}{T}\int
\frac{1}{n}\sum\limits_{i=1}^{n}
\frac{1}{R^3}\Theta_{\bfx_i,R}(t)\,dt\;.
\end{equation}

Denote by
\begin{equation}\label{e_R_def}
\bfe_R=\lgl \bfe_{\bfx,R}(t)\rgl_R\;,
\end{equation}
\begin{equation}\label{E_R_def}
\bfE_R=\lgl \bfE_{\bfx,R}(t)\rgl_R\;,
\end{equation}
\begin{equation}\label{Phi_R_def}
\Phi_R=\lgl \Phi_{\bfx,R}(t)\rgl_R\;,
\end{equation}
the averaged localized energy, enstrophy and inward-directed
flux over balls of radius $R$ covering $B(\bfo,R_0)$.

Also, introduce the time-space average of the localized energy on
$B(\bfo,R_0)$,
\begin{equation}\label{e_def}
\bfe=\frac{1}{T}\int
\frac{1}{R_0^3}\bfe_{\bfo,R_0}(t)\,dt=\frac{1}{T}\frac{1}{R_0^3}\iint \frac{1}{2}|\bfu|^2
\phi_0^{2\delta-1}\,d\bfx\,dt\;
\end{equation}
and the time-space average of the localized enstrophy on $B(\bfo,R_0)$,
\begin{equation}\label{E_def}
{\bfE}=\frac{1}{T}\int
\frac{1}{R_0^3}\bfE_{\bfo,R_0}(t)\,dt=\frac{1}{T}\frac{1}{R_0^3}
\iint |\nabla\otimes\bfu|^2\phi_0\,d\bfx\,dt\;
\end{equation}
where
\begin{equation}\label{phi0}
\phi_0(t,\bfx)=\eta(t)\psi_0(\bfx)
\end{equation}
with $\psi_0$ defined in (\ref{psi0}).

Finally, define Taylor length scale associated with $B(\bfo,R_0)$ by
\begin{equation}\label{tau_def}
\tau_0=\left(\frac{\bfe}{\bfE}\right)^{1/2}\;.
\end{equation}

Note that the possible lack of regularity may produce additional
loss of energy, resulting in {\em anomalous} energy dissipation and
the loss of flux leading to the strict inequality in
(\ref{loc_ene_ineq}). Let us mention here that in the turbulence
literature the term `anomalous dissipation' is usually utilized in
the context of the possible energy dissipation due to the (possible)
singularities in the 3D Euler equations (the observation originally
made by Onsager); for rigorous results on Onsager's conjecture on
the energy conservation in the Euler equations see, e.g.,
\cite{CET}, and a recent work \cite{CCFS}.

Denote by $\Phi_{\bfx_0,R}^{\infty}$ the loss of flux due to
possible singularities in $[0,2T]\times B(\bfx_0,2R)$,
\begin{equation}\label{ener_eq1}
\begin{aligned}
&\iint(\frac{1}{2}|\bfu|^2+p)\bfu\cdot\nabla\phi\,d\bfx\,dt\,- \Phi_{\bfx_0,R}^{\infty}\\
&= \nu\iint|\nabla\otimes\bfu|^2\phi\,d\bfx\,dt\, -
\frac{1}{2}\iint|\bfu|^2(\partial_t\phi+\nu\Delta\phi)\,d\bfx\,dt\;
\end{aligned}
\end{equation}
where $\phi=\eta\psi$ with $\eta$ and $\psi$ as in (\ref{eta_def}) and
(\ref{psi_def}-\ref{psi_def_add2}). In particular, denote by
$\Phi_{\infty}=\Phi_{\bfo,R_0}^{\infty}$ the loss of flux due to
singularities in $[0,2T]\times B(\bfo,2R_0)$.

We will also consider the time-space ensemble averages of these {\em
anomalous} fluxes,

\begin{equation}\label{Phi_R_inf_def}
\Phi_R^{\infty}=\frac{1}{n}\sum\limits_{i=1}^{n}
\frac{1}{T}\frac{1}{R^3}\Phi_{\bfx_i,R}^{\infty}\;.
\end{equation}

Note that due to (\ref{loc_ene_ineq}), all the anomalous fluxes are
nonnegative,
\begin{equation}
\Phi_{\bfx_0,R}^{\infty}\ge0,\qquad\Phi_{\infty}\ge0,\qquad\Phi_R^{\infty}\ge0\;;
\end{equation}
they are all zero provided the equality holds in
(\ref{loc_ene_ineq}) inside $[0,2T]\times B(\bfo,2R_0)$. In particular,
the anomalous fluxes are all zero provided the solution in view is
{\em regular} on $[0,2T]\times B(\bfo,2R_0)$.

Consequently, the total localized flux into $B(\bfx_0 ,R)$ over
interval $[0,2T]$, including the (loss of) flux due to the possible
loss of regularity, is
\begin{equation}\label{Psi_R_x0_def}
\Psi_{\bfx_0,R}=\int{\Phi_{\bfx_0,R}(t)\,dt}-\Phi_{\bfx_0,R}^{\infty}
\end{equation}
and the time-space ensemble average of this flux 
at scales $R$ over time $T$ is
\begin{equation}\label{Psi_R_def}
\Psi_R=\Phi_R-\Phi_{R}^{\infty}\;.
\end{equation}
We will refer to $\Psi_{\bfx_0,R}$ and $\Psi_R$ as the {\em
modified} flux over $[0,2T]$ into $B(\bfx_0,R)$ and the (time-space
ensemble) averaged  {\em modified} flux at the scale $R$,
respectively.

Let $K_1,K_2>1$ be two positive integers (independent of $R,R_0$,
and any of the parameters of the NSE).

\begin{defin}
We say that a covering of $B(\bfo,R_0)$  by $n$ balls of radius $R$
is {\em optimal} if
\begin{equation}\label{n_con1}
\left(\frac{R_0}{R}\right)^3\le n\le K_1\left(\frac{R_0}{R}\right)^3;
\end{equation}
\begin{equation}\label{n_con2}
\mbox{any}\  \bfx\in B(\bfo,R_0)\  \mbox{is covered by at most}\ K_2 \
\mbox{balls}\ B(\bfx_i,2R)\,.
\end{equation}
\end{defin}

Note that optimal coverings exist for any $0<R\le R_0$ provided
$K_1$ and $K_2$ are large enough. In fact, the choice of $K_1$ and
$K_2$ depends only on dimension of the space; in $\mR3$ we can
choose $K_1=K_2=8$.

Henceforth, we assume that the averages $\lgl\cdot\rgl_R$ are taken
with respect to optimal coverings.

\begin{lem}\label{anom_flux_lem}
If the covering $\{B(\bfx_i,R)\}_{i=1,n}$ of $B(\bfo,R_0)$ is optimal then
\begin{equation}\label{Phi_infty_est}
\Phi_R^{\infty}\le K\frac{1}{T}\frac{1}{R_0^3}\Phi_{\infty}\;
\end{equation}
where $K>0$ is a constant depending only on $K_2$ and dimension of
the space $\mR3$.
\end{lem}

\begin{proof}
Let $\{\bfx_{i_j}\}$ be a subset of $\{\bfx_i\}_{i=1,n}$ such that
interiors of the balls $B(\bfx_{i_j},2R)$ are pairwise disjoint.
Using (\ref{ener_eq1}), we obtain
\begin{equation}\label{lem_eq1}
\begin{aligned}
&\iint(\frac{1}{2}|\bfu|^2+p)\bfu\cdot\nabla\phi_0\,d\bfx\,dt\,- \Phi_{\infty}\\
&=
\nu\iint|\nabla\otimes\bfu|^2\phi_0\,d\bfx\, dt\, -
\frac{1}{2}\iint|\bfu|^2(\partial_t\phi_0+\nu\Delta\phi_0)\,d\bfx\,dt\;
\end{aligned}
\end{equation}
and
\begin{equation}\label{lem_eq2}
\begin{aligned}
&\iint(\frac{1}{2}|\bfu|^2+p)\bfu\cdot\nabla(\sum\limits_{j}\phi_{i_j})\,d\bfx\,dt\,-
\sum\limits_{j}\Phi_{\bfx_{i_j},R}^{\infty}\\
&= \nu\iint|\nabla\otimes\bfu|^2(\sum\limits_{j}\phi_{i_j})d\bfx dt
- \frac{1}{2}\iint|\bfu|^2[\partial_t(\sum\limits_{j}\phi_{i_j})+
\nu\Delta(\sum\limits_{j}\phi_{i_j})]d\bfx dt
\end{aligned}
\end{equation}
where $\phi_0=\eta\psi_0$ and $\phi_{i_j}=\eta\psi_{i_j}$ with
$\eta$ as in (\ref{eta_def}), $\psi_0$ as in (\ref{psi0}) and
$\psi_{i_j}$ a test function corresponding to $B(\bfx_{i_j},R)$
satisfying (\ref{psi_def}-\ref{psi_def_add2}).

Note that the definitions of $\phi_0$ and $\phi_{i_j}$ imply
\[
\tilde{\phi}=\phi_0-\sum\limits_{j}\phi_{i_j}\ge0\;;
\]
hence, by the local energy inequality (\ref{loc_ene_ineq}),
\begin{equation}\label{lem_eq3}
\begin{aligned}
&\iint(\frac{1}{2}|\bfu|^2+p)\bfu\cdot\nabla\tilde{\phi}\,d\bfx\,dt\,\\
&
\ge
\nu\iint|\nabla\otimes\bfu|^2\tilde{\phi}\,d\bfx\, dt\, -
\frac{1}{2}\iint|\bfu|^2(\partial_t\tilde{\phi}+\nu\Delta\tilde{\phi})\,d\bfx\,dt\;.
\end{aligned}
\end{equation}
If we add relations (\ref{lem_eq2}) and (\ref{lem_eq3}) and then subtract
(\ref{lem_eq1}) we obtain
\begin{equation}\label{claim1}
\Phi_{\infty}\ge\sum\limits_{j}\Phi_{\bfx_{i_j},R}^{\infty}\;.
\end{equation}

Let $\mathcal{L}$ be a cubic lattice inside $B(\bfo,R_0)$ with the
points situated at the vertices of cubes of side $R/2$ (Note that
this lattice can be chosen such that the number of points in it is
between $2^3(R_0/R)^3$ and $(4\pi/3)2^3(R_0/R)^3$).

Since the covering $\{B(\bfx_i,R)\}$ is optimal, each point in
$\mathcal{L}$ is contained in at most $K_2$ balls. Moreover, any
ball in the covering will contain at least one point from the
lattice.

If $\mathcal{L}'$ is sub-lattice of $\mathcal{L}$ with points at
vertices of cubes of side $4R$, then the interiors of balls of
radius $2R$ containing different points of $\mathcal{L}'$ are
pairwise disjoint, and thus if we denote by $B(\bfx_{i_p},R)$ a ball
from the covering $\{B(\bfx_i,R)\}$ containing the point
$p\in\mathcal{L}'$, by (\ref{lem_eq3}),
\[\Phi_{\infty}\ge\sum\limits_{p\in\mathcal{L}'}\Phi_{\bfx_{i_p},R}^{\infty}\;.\]

Note that for each point $p\in\mathcal{L}'$ there are at most $K_2$ choices for
$B(\bfx_{i_p},R)$.
So
\[
K_2\Phi_{\infty}\ge\sum\limits_{i:B(\bfx_{i},R)\cap\mathcal{L}'
\not=\emptyset}\Phi_{\bfx_{i},R}^{\infty}\;.
\]
Clearly $\mathcal{L}$  can be written as a union of $8^3=256$
sub-lattices $\mathcal{L}'_k$, $k=1,\dots,256$, each
$\mathcal{L}'_k$ having the same properties as $\mathcal{L}'$.
Thus,
\[
8^3K_2\Phi_{\infty}\ge\sum\limits_{i=1}^{n}\Phi_{\bfx_{i},R}^{\infty}\;.
\]
Consequently,
\[\Phi_R^{\infty}=\frac{1}{T}\frac{1}{R^3}\frac{1}{n}\sum\limits_{i=1}^{n}\Phi_{\bfx_{i},R}^{\infty}\le
8^3K_2\frac{1}{T}\frac{1}{R^3}\frac{1}{n}\Phi_{\infty}\le
8^3K_2\frac{1}{T}\frac{1}{R_0^3}\Phi_{\infty}\;\] where the last
inequality is due to $n$ satisfying (\ref{n_con2}).
\end{proof}

According to the lemma, the time-space ensemble averages
$\Phi_R^{\infty}$ taken over the optimal coverings at the scale $R$
are bounded, \emph{independently} of $R$, by the average loss of
flux due to possible singularities inside $B(\bfo,2R_0)$.

\section{Energy Cascade}\label{balls}

Let $\{B(\bfx_i,R)\}_{i=1,n}$ be an optimal covering of $B(\bfo,R_0)$.

Note that the local energy equality (\ref{ener_eq1}) and the
definitions of $\bfE_R$, $\Phi_R$ and $\Phi_R^{\infty}$
(\,(\ref{E_R_def}), (\ref{Phi_R_def}) and (\ref{Phi_R_inf_def})\,)
imply

\begin{equation}\label{ene-eq}
\Psi_R=\Phi_R-\Phi_R^{\infty}= \nu \bfE_R -
\frac{1}{n}\sum\limits_{i=1}^{n}\frac{1}{T}\frac{1}{R^3}\iint\frac{1}{2}
|\bfu|^2(\partial_t\phi_i+\nu\Delta\phi_i)\,d\bfx\,dt\;
\end{equation}
where $\phi_i=\eta\psi_i$ and $\psi_i$ is the spatial cut-off on
$B(\bfx_i,2R)$ satisfying (\ref{eta_def}-\ref{psi_def_add2}).

If
\begin{equation}\label{T_con}
T\ge \frac{R_0^2}{\nu},
\end{equation}
then for any $0<R\le R_0$,

\[|(\phi_i)_t|=|\eta_t\psi_i|\le C_0\frac{1}{T}\eta^{\delta}\psi_i\le
\nu\frac{C_0}{R^2}\phi_i^{2\delta-1},\]
\[\nu|\Delta\phi_i|=\nu|\eta\Delta\psi_i|\le C_0\frac{\nu}
{R^2}\eta\psi_i^{2\delta-1}\le\nu\frac{C_0}{R^2}\phi_i^{2\delta-1};\]
hence,
\[\Psi_R\ge \nu \bfE_R -\nu \frac{C_0}{R^2}\,\bfe_R.\]

The optimality conditions (\ref{n_con1}) and (\ref{n_con2}) paired
with (\ref{psi_def}--\ref{psi_def_add2}) imply
\begin{equation}\label{E_R_E_ineq}
\bfE_R\ge\frac{1}{K_1}\bfE
\end{equation}
and
\begin{equation}\label{e_R_e_ineq}
\bfe_R\le{K_2}\bfe\;.
\end{equation}

Consequently,

\begin{equation}\label{low_bd_rel}
\Psi_R\ge \nu \frac{1}{K_1}\bfE -\nu \frac{C_0K_2}{R^2}\,\bfe\;
\end{equation}
leading to the following proposition.

\begin{prop}
\begin{equation}\label{low_bd}
\Psi_R\ge c_1\nu \bfE\,\left(1-c_2\frac{\tau_0^2}{R^2}\right)
\end{equation}
with $c_1=1/K_1$ and $c_2=C_0K_1K_2$ (provided conditions
(\ref{n_con1}-\ref{n_con2}) are satisfied).
\end{prop}

Suppose that
\begin{equation}\label{scales_con}
\tau_0< \frac{\gamma}{c_2^{1/2}}R_0
\end{equation}
for some $0<\gamma<1$. Then, for any $R$,
$(c_2^{1/2}/\gamma)\,\tau_0 \le R \le R_0$,
\begin{equation}\label{lower_bd}
\Psi_R\ge{c_1}(1-\gamma^2)\nu \bfE=c_{0,\gamma}\nu \bfE\;
\end{equation}
where
\begin{equation}
c_{0,\gamma}={c_1}(1-\gamma^2)=\frac{1-\gamma^2}{K_1}\;.
\end{equation}

To obtain an upper bound on the averaged modified flux, note that
for optimal coverings, in addition to (\ref{e_R_e_ineq}),
\begin{equation}\label{E_R_le_E_ineq}
\bfE_R\le{K_2}\bfE\;.
\end{equation}
Hence, (\ref{ene-eq}) implies
\[\Psi_R\le \nu \bfE_R+\nu \frac{C_0}{R^2}\bfe_R\le\nu K_2\bfE+\nu C_0K_2\frac{1}{R^2}\,\bfe.\]
If the condition (\ref{scales_con}) holds for some $0<\gamma<1$,
then it follows that for any $R$, $({c_2}^{1/2}/{\gamma})\,\tau_0
\le R\le R_0$,
\begin{equation}
\Psi_R\le \nu K_2 \bfE+\nu \frac{C_0K_2\gamma^2}{c_2}\bfE\le
c_{1,\gamma}\nu \bfE\;
\end{equation}
where
\begin{equation}
c_{1,\gamma}=K_2 \left[1+\frac{C_0\gamma^2}{c_2}\right]=K_2
\left[1+\frac{\gamma^2}{K_1K_2}\right]\;.
\end{equation}

Thus we have proved the following.

\begin{thm}\label{balls_thm}
Assume that for some $0<\gamma<1$
\begin{equation}\label{scales_con_fin}
\tau_0< c{\gamma}\,R_0\;,
\end{equation}
where
\begin{equation}\label{c_con1}
c=\frac{1}{\sqrt{C_0K_1K_2}}\;.
\end{equation}
Then, for all $R$,
\begin{equation}\label{inert_range}
\frac{1}{c\gamma}\,\tau_0\le R\le R_0,
\end{equation}
the averaged modified flux $\Psi_R$ satisfies
\begin{equation}\label{ener_casc}
c_{0,\gamma}\nu \bfE\le\Psi_R\le c_{1,\gamma} \nu \bfE\;
\end{equation}
where
\begin{equation}\label{c_con2}
c_{0,\gamma}=\frac{1-\gamma^2}{K_1}\,, \quad
c_{1,\gamma}=K_2 \left[1+\frac{\gamma^2}{K_1K_2}\right]\;,
\end{equation}
and the average $\lgl\cdot\rgl_R$ is computed over a time interval
$[0,2T]$ with $T\ge R_0^2/\nu$ and determined by an optimal covering
of $B(\bfo,R_0)$ (i.e., a covering satisfying (\ref{n_con1}) and
(\ref{n_con2})).
\end{thm}

\begin{obs}{\em
As noted in the introduction -- in the case the global energy is
non-increasing -- the theorem provides a sufficient condition for
the energy cascade; i.e., a nearly constant nonlinear transfer of
time-averaged (kinetic) energy to smaller scales across the inertial
range defined by (\ref{inert_range}).

More precisely, since we are working with weak solutions,
the expression for the rate of change of local kinetic
energy (\ref{loc_tr}) morphs into
\begin{equation}
 -\iint \frac{1}{2}|\bfu|^2 \phi_t = \Psi_{\bfx_0,R} + \ \mbox{viscous terms}.
\end{equation}
The interpretation remains the same, the only differences being that the
time-derivative of the local kinetic energy is taken in the sense of distributions
and the flux got replaced with the modified flux to account for
possible singularities.

Note that the averages are taken over finite-time
intervals
 $[0,2T]$ with $T\ge R_0^2/\nu$ (see (\ref{T_con})\,). This lower
 bound on the length of the time interval $T$ is consistent with the
 picture of decaying turbulence; namely, small $\nu$ corresponds to
 the well-developed turbulence which then persists for a longer time
 and it makes sense to average over longer time-intervals.
}\end{obs}

\begin{obs}\label{Rem_4.2}{\em
In the language of turbulence, the condition (\ref{scales_con_fin})
simply reads that the Taylor \emph{micro-scale} computed over the
domain in view is smaller than the \emph{integral scale} (diameter
of the domain).

On the other hand, (\ref{scales_con_fin}) is equivalent to

\[
\frac{1}{T}\iint |\bfu|^2\phi_0^{2\delta-1}\,d\bfx\,dt
<\frac{\gamma^2}{C_0K_1K_2}{R_0^2} \frac{1}{T}\iint
|\nabla\otimes\bfu|^2\phi_0\,d\bfx\,dt\;
\]
which can be read as a requirement that the time average of a
Poincar\'e-like inequality on $B(\bfo,2R_0)$ is not saturating; this
will hold for a variety of flows in the regions of active fluid
(large gradients). }\end{obs}

\begin{obs}\label{Phi_low_bds}
{\em

Since $\Phi_R^{\infty}\ge 0$, all the {\em lower} bounds on the
averaged modified fluxes hold for the usual averaged fluxes
$\Phi_R$; in particular,
\begin{equation}
\Phi_R\ge c_1\nu \bfE\,\left(1-c_2\frac{\tau_0^2}{R^2}\right)
\end{equation}
and consequently, provided (\ref{scales_con_fin}) and
(\ref{inert_range}) hold,
\begin{equation}
\Phi_R\ge c_{0,\gamma}\nu\bfE\;.
\end{equation}

Also, if a solution $\bfu$ is such that the (\ref{loc_ene_ineq})
holds with equality and in particular, if the solution is
\emph{regular}, then \emph{all} the estimates, including
(\ref{ener_casc}), hold for the usual averaged flux $\Phi_R$.

}
\end{obs}

\begin{obs}\label{E'_rem1}{\em
If we do not impose the additional assumptions (\ref{psi_def_add1})
and (\ref{psi_def_add2}) for the test functions on the balls
$B(\bfx_i,R)\not\subset B(\bfo,R_0)$, then the lower bounds for
$\Psi_R$ in (\ref{low_bd}) and (\ref{ener_casc}) will hold with $\bfE$
replaced by the time-space average of  the {\em{non-localized}}
in space enstrophy on $B(\bfo,R_0)$,
\[
E'=\frac{1}{T}\int\limits_0^{2T}
\frac{1}{R_0^3}\int\limits_{B(\bfo,R_0)}|\nabla\otimes\bfu|^2\eta\,d\bfx\,dt\;.
\]
This is the case because the estimate (\ref{E_R_E_ineq}) gets
replaced with
\[\bfE_R\ge\frac{1}{K_1}E'\;.\]
}\end{obs}

\section{Locality of the averaged flux}

Let $\bfx_0\in B(\bfo,R_0)$, $0<R_2<R_1\le R_0$. In order to study
the flux through the shell $A(\bfx_0,R_1,R_2)$ between the spheres
$S(\bfx_0,R_2)$ and $S(\bfx_0,R_1)$, in what follows, we will
consider the modified test functions
$\phi=\phi_{\bfx_0,T,R_1,R_2}(t,\bfx)=\eta(t)\psi(\bfx)$ to be used
in the local energy inequality (\ref{loc_ene_ineq}) where
$\eta=\eta_T(t)$ as in (\ref{eta_def}) and
$\psi \in \mathcal{D}(A(\bfx_0,2R_1,R_2/2))$ satisfying
\begin{equation}\label{psi_def2}\begin{aligned}
&0\le\psi\le\psi_0,\quad\psi=1\ \mbox{on}\
A(\bfx_0,R_1,R_2),
&\frac{|\nabla\psi|}{\psi^{\delta}}\le\frac{C_0}{\tilde{R}}, \
\frac{|\triangle\psi|}{\psi^{2\delta-1}}\le\frac{C_0}{\tilde{R}^2}\;
\end{aligned}
\end{equation}
where $\psi_0$ is defined in (\ref{psi0}) and
\begin{equation}\label{tilde_R}
\tilde{R}=\tilde{R}(R_1,R_2)=\min\{R_2,R_1-R_2\}\;.
\end{equation}


Use the $\phi$ above to define the localized time-averaged flux
through the shell between the spheres $S(\bfx_0,R_2)$ and
$S(\bfx_0,R_1)$ as
\begin{equation}
\Phi_{\bfx_0,R_1,R_2}=\frac{1}{T}\iint
(\frac{1}{2}|\bfu|^2+p)
\bfu\cdot\nabla\phi\,d\bfx\,dt\;.
\end{equation}

If $\Phi_{\bfx_0,R_1,R_2}^{\infty}$ is the anomalous flux inside
$A(\bfx_0,2R_1,R_2/2)$, i.e., if $\Phi_{\bfx_0,R_1,R_2}$ satisfies
\begin{equation}\label{ener_eq2}
\begin{aligned}
&\iint(\frac{1}{2}|\bfu|^2+p)\bfu\cdot\nabla\phi\,d\bfx\,dt\,- \Phi_{\bfx_0,R_1,R_2}^{\infty}\\
&
=
\nu\iint|\nabla\otimes\bfu|^2\phi\,d\bfx\,dt\, -
\frac{1}{2}\iint|\bfu|^2(\partial_t\phi+\nu\Delta\phi)\,d\bfx\,dt\;,
\end{aligned}
\end{equation}
then define the time average of the modified localized fluxes
through the shells $A(\bfx_0,R_1,R_2)$ by
\begin{equation}
\Psi_{\bfx_0,R_1,R_2}=\Phi_{\bfx_0,R_1,R_2}-\frac{1}{T}\Phi_{\bfx_0,R_1,R_2}^{\infty}\;.
\end{equation}
As already mentioned, the modified fluxes can be viewed as total
fluxes including the (loss of) flux due to possible singularities
inside the shell. Also note that the local energy inequality
(\ref{loc_ene_ineq}) implies
\begin{equation}
\Phi_{\bfx_0,R_1,R_2}^{\infty}\ge 0.
\end{equation}

Define the time-averaged energy and enstrophy per unit of mass in the shell between
the spheres $S(\bfx_0,R_2)$ and $S(\bfx_0,R_1)$ by
\begin{equation}
\begin{aligned}
&\bfe_{\bfx_0,R_1,R_2}=\frac{1}{T}\iint\frac{1}{2}
|\bfu|^2\phi^{2\delta-1}\,d\bfx\,dt\;,\\
&\bfE_{\bfx_0,R_1,R_2}=\frac{1}{T}\iint
|\nabla\otimes\bfu|^2\phi\,d\bfx\,dt\;;
\end{aligned}
\end{equation}
then,
\begin{equation}
\tau_{\bfx_0,R_1,R_2}=\left(\frac{\bfe_{\bfx_0,R_1,R_2}}{\bfE_{\bfx_0,R_1,R_2}}\right)^{1/2}
\end{equation}
is the \emph{local} Taylor length scale associated with the shell
$A(\bfx_0,R_1,R_2)$.

Note that
\begin{equation}\label{phi_est1}
\nu|\Delta\phi|=\nu|\eta\Delta\psi|\le \nu\frac{C_0}{\tilde{R}^2}|\eta\phi^{2\delta-1}|\le
\nu\frac{C_0}{\tilde{R}^2}\phi^{2\delta-1}
\end{equation}
and
\begin{equation}\label{phi_est2}
|\phi_t|=|\eta_t\psi|\le C_0\frac{1}{T}\eta^{\delta}\psi\le\nu\frac{C_0}{\tilde{R}^2}\phi^{2\delta-1},
\end{equation}
provided
\begin{equation}\label{T_con2}
T\ge \frac{R_0^2}{\nu}\, \left(\ge\frac{\tilde{R}^2}{\nu}\right)\;.
\end{equation}

Hence, (\ref{ener_eq2}) implies that for any $\bfx_0\in B(\bfo,R_0)$
and any $0<R_2<R_1\le R_0$,
\begin{equation}
\begin{aligned}
\Psi_{\bfx_0,R_1,R_2} &\ge \nu \bfE_{\bfx_0,R_1,R_2}-\nu \frac{C_0}{\tilde{R}^2}\bfe_{\bfx_0,R_1,R_2}\\
&= \nu
\bfE_{\bfx_0,R_1,R_2}\,\left(1-C_0\frac{\tau^2_{\bfx_0,R_1,R_2}}{\tilde{R}^2}\right)\;
\end{aligned}
\end{equation}
leading to the following proposition.

\begin{prop}
Let $0<\gamma<1$. Then, for any shell $A(\bfx_0,R_1,R_2)$ satisfying
\begin{equation}\label{scales_con2}
\tau_{\bfx_0,R_1,R_2}<\frac{\gamma}{C_0^{1/2}}\tilde{R}\;
\end{equation}
with $\tilde{R}$ defined by (\ref{tilde_R}),
\begin{equation}\label{low_bd2}
\Psi_{\bfx_0,R_1,R_2}\ge
\nu \bfE_{\bfx_0,R_1,R_2}(1-\gamma^2)\;.
\end{equation}
\end{prop}

Similarly, utilizing (\ref{ener_eq2}) again, we obtain an upper
bound
\[
\Psi_{\bfx_0,R_1,R_2}\le\nu \bfE_{\bfx_0,R_1,R_2}+\nu\frac{C_0}{\tilde{R}^2}\bfe_{\bfx_0,R_1,R_2}\;.
\]

If the condition (\ref{scales_con2}) holds for some $0<\gamma<1$, then it follows that
\begin{equation}
\Psi_{\bfx_0,R_1,R_2}\le \nu \bfE_{\bfx_0,R_1,R_2}+\gamma^2\nu
\bfE_{\bfx_0,R_1,R_2}\le (1+\gamma^2)\,\nu \bfE_{\bfx_0,R_1,R_2}\;;
\end{equation}
thus, we have arrived at our first locality result.

\begin{thm}\label{shells_thm1}
Let $0<\gamma<1$, $\bfx_0\in B(\bfo,R_0)$ and $0<R_2<R_1\le R_0$. If
\begin{equation}\label{scales_con2_fin}
\tau_{\bfx_0,R_1,R_2}<\frac{\gamma}{C_0^{1/2}}\tilde{R}\;
\end{equation}
with $\tilde{R}$ defined by (\ref{tilde_R}), then
\begin{equation}\label{ener_casc1}
(1-\gamma^2)\,\nu \bfE_{\bfx_0,R_1,R_2}\le\Psi_{\bfx_0,R_1,R_2}\le
(1+\gamma^2)\,\nu \bfE_{\bfx_0,R_1,R_2}\;
\end{equation}
where the time average is taken over an interval of time $[0,2T]$
with $T\ge R_0^2/\nu$.
\end{thm}

\begin{obs}{\em
The theorem states that if the local Taylor scale
$\tau_{\bfx_0,R_1,R_2}$, associated with a shell
$A(\bfx_0,R_1,R_2)$, is smaller than the thickness of the shell
$\tilde{R}$ (a local integral scale), then the time average of the
modified flux through that shell towards its center $\bfx_0$ is
comparable to the time average of the localized enstrophy in the
shell, $\bfE_{\bfx_0,R_1,R_2}$. Thus, under the assumption
(\ref{scales_con2_fin}) the flux through the shell
$A(\bfx_0,R_1,R_2)$ depends essentially only on the enstrophy
contained in the neighborhood of the shell, regardless of what
happens at the other scales, making (\ref{scales_con2}) a sufficient
condition for the \emph{locality} of the flux through
$A(\bfx_0,R_1,R_2)$. }\end{obs}

\begin{obs}{\em
Similarly as in the case of condition (\ref{scales_con_fin}), we can
observe that condition (\ref{scales_con2_fin}) can be viewed as a
requirement that the time average of a Poincar\'e-like inequality on
the shell is not saturating making it plausible in the case of
intense fluid activity in a neighborhood of the shell. }\end{obs}

\begin{obs}\label{Phi_low_bds2}
{\em Since $\Phi_{\bfx_0,R_1,R_2}^{\infty}\ge 0$, all the {\em
lower} bounds on the modified fluxes hold for the usual fluxes
$\Phi_{\bfx_\bfo,R_0,R_1}$; in particular, we have
\begin{equation}
\Phi_{\bfx_0,R_1,R_2}\ge \nu \bfE_{\bfx_0,R_1,R_2}\,
\left(1-C_0\frac{\tau^2_{\bfx_0,R_1,R_2}}{\tilde{R}^2}\right)\;
\end{equation}
and, provided (\ref{scales_con2_fin}) holds,
\begin{equation}
\Phi_{\bfx_\bfo,R_0,R_1}\ge (1-\gamma^2)\,\nu \bfE_{\bfx_0,R_1,R_2}\;.
\end{equation}

Also, if a solution $\bfu$ is such that the (\ref{loc_ene_ineq})
holds with equality and in particular, if the solution is {\em
regular}, then \emph{all} the estimates, including
(\ref{ener_casc1}), hold for the usual averaged flux
$\Phi_{\bfx_0,R_1,R_2}$. }
\end{obs}

In order to further study the locality of the flux, we will estimate
the ensemble averages of the fluxes through the shells
$A(\bfx_i,2R,R)$ of thickness $\tilde{R}=R$. Since we are interested
in the shells inside $B(\bfo,R_0)$, we require the lattice points
$\bfx_i$ to satisfy
\begin{equation}\label{x_i_shells_con}
B(\bfx_i,R)\subset B(\bfo,R_0)\;.
\end{equation}
To each $A(\bfx_i,2R,R)$ we associate a test function
$\phi_i=\eta\psi_i$ where $\eta$ satisfies (\ref{eta_def}) and
$\psi_i$ satisfies the following.

If $A(\bfx_i,2R,R)\subset B(\bfo,R_0)$, then
$\psi_i\in\mathcal{D}(A(\bfx_i,4R,R/2))$ with
\begin{equation}\label{psi_shells_def}\begin{aligned}
\quad 0\le\psi_i\le\psi_0, \ \psi_i=1\ \mbox{on}\ A(\bfx_i,2R,R)\cap
B(\bfo,R_0), \
\frac{|\nabla\psi_i|}{\psi_i^{\delta}}\le\frac{C_0}{R}, \
\frac{|\triangle\psi_i|}{\psi_i^{2\delta-1}}\le\frac{C_0}{R^2}\;,
\end{aligned}
\end{equation}
and if  $A(\bfx_i,2R,R)\not\subset B(\bfo,R_0)$ (i.e. we have $B(\bfx_i,R)\subset B(\bfo,R_0)$
and $B(\bfx_i,2R)\setminus B(\bfo,R_0)\not=\emptyset$), then
$\psi_i\in\mathcal{D}(B(\bfo,2R_0))$ with $\psi_i=1\ \mbox{on}\ A(\bfx_0,2R,R)
\cap B(\bfo,R_0)$ satisfying, in addition to (\ref{psi_shells_def}), the
following:
\begin{equation}\label{psi_shells_def_add1}
\begin{aligned}
&
\psi_i=\psi_0\ \mbox{on the part of the cone in}\ \mR3\ \mbox{centered at
zero and passing}\\ &\mbox{ through}\   S(\bfo,R_0)\cap B(\bfx_i,2R)\
\mbox{between}\  S(\bfo,R_0)\ \mbox{and}\
S(\bfo,2R_0)
\end{aligned}
\end{equation}
and
\begin{equation}\label{psi_shells_def_add2}
\begin{aligned}
&
\psi_i=0\ \mbox{on}\ B(\bfo,R_0)\setminus A(\bfx_i,4R,R/2)\ \mbox{and outside the part of the}\\ &
\mbox{cone in}\ \mR3\ \mbox{centered at zero and passing through}\ S(\bfo,R_0)\cap B(\bfx_i,4R)\\
 &
 \mbox{between}\  S(\bfo,R_0)\ \mbox{and}\ S(\bfo,2R_0).
\end{aligned}
\end{equation}

Figure \ref{shall_fig} illustrates the definition of $\psi_i$ in the case
$A(\bfx_i,2R,R)$ is not entirely contained in
$B(\bfo,R_0)$.

\begin{figure}
  \centerline{\includegraphics[scale=1, viewport=173 461 509 669, clip] {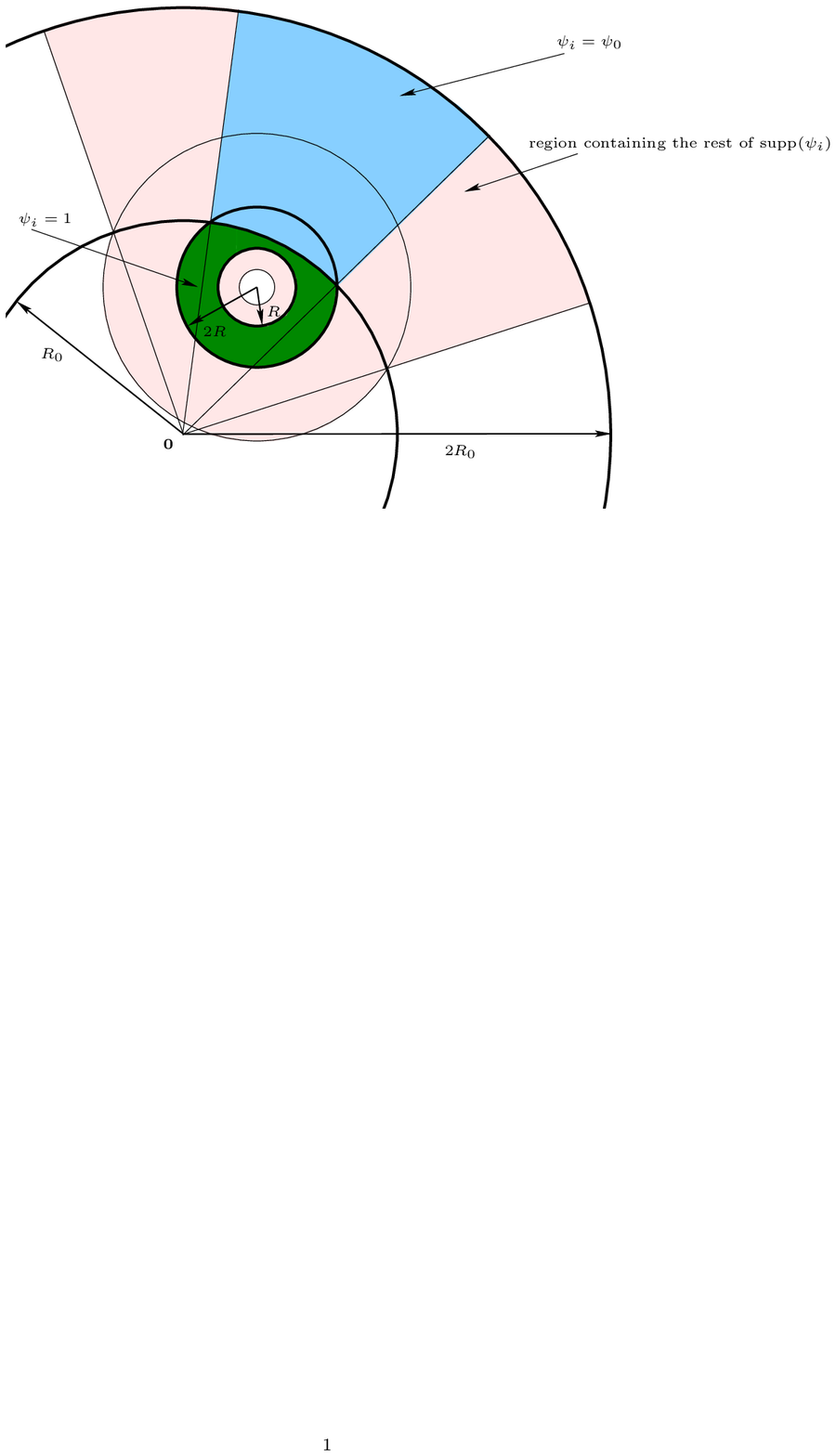}}
  \caption{A cross-section of regions of supp$(\psi_i)$ in the case $A(\bfx_i,2R,R)\not\subset B(\bfo,R_0)$.}
  \label{shall_fig}
\end{figure}

Similarly as in the previous section, we consider {\em optimal}
coverings of $B(\bfo,R_0)$ by shells $\{A(\bfx_i,2R,R)\}_{i=1,n}$
such that (\ref{x_i_shells_con}) is satisfied,
\begin{equation}\label{shells_n_con1}
\left(\frac{R_0}{R}\right)^3\le n\le K_1\left(\frac{R_0}{R}\right)^3
\end{equation}
and
\begin{equation}\label{Shells_n_con2}
\mbox{any}\  \bfx\in B(\bfo,R_0)\  \mbox{is covered by at most}\  K_2\
\mbox{shells}\ A(\bfx_i,4R,R/2)\;.
\end{equation}

Introduce
\begin{equation}
\tilde{\bfe}_{2R,R}=\frac{1}{n}\sum\limits_{i=1}^{n}e_{\bfx_i,2R,R}\;,
\end{equation}
\begin{equation}
\tilde{\bfE}_{2R,R}=\frac{1}{n}\sum\limits_{i=1}^{n}E_{\bfx_i,2R,R}\;,
\end{equation}
and
\begin{equation}
\tilde{\Phi}_{2R,R}=\frac{1}{n}\sum\limits_{i=1}^{n}\Phi_{\bfx_i,2R,R}\;,
\end{equation}
the ensemble averages of the time-averaged energy, enstrophy, and
flux on the shells of thickness $R$ corresponding to the covering
$\{A(\bfx_i,2R,R)\}_{i=1,n}$\,.

The ensemble average of the time-averaged modified flux
on shells of thickness ${R}$ is then defined by
\begin{equation}
\tilde{\Psi}_{2R,R}=\frac{1}{n}\sum\limits_{i=1}^{n}\Psi_{\bfx_i,2R,R}=
\tilde{\Phi}_{2R,R}- \tilde{\Phi}_{2R,R}^{\infty}\;
\end{equation}
where
\begin{equation}
\tilde{\Phi}_{2R,R}^{\infty}=\frac{1}{n}\sum\limits_{i=1}^{n}\frac{1}{T}\Phi_{\bfx_i,2R,R}^{\infty}\;
\end{equation}
is the ensemble average of the time-averaged anomalous fluxes through the shells
of thickness $R$ inside $B(\bfo,R_0)$.

An argument analogous to the one in Lemma \ref{anom_flux_lem} implies that if the
covering of $B(\bfo,R_0)$ is optimal, then
\begin{equation}
\tilde{\Phi}_{2R,R}^{\infty}\le K\frac{1}{T}\Phi_{\infty}\;.
\end{equation}

Taking the time ensemble averages in (\ref{ener_eq2}) and applying
the bounds (\ref{phi_est1}-\ref{phi_est2}), we arrive at
\begin{equation}\label{low_R12_bd}
\tilde{\Psi}_{2R,R}\ge \nu\tilde{\bfE}_{2R,R}-\nu\frac{C_0}{R^2}\,\tilde{\bfe}_{2R,R}\;.
\end{equation}

If the covering is optimal, i.e., if
(\ref{x_i_shells_con}) and (\ref{shells_n_con1}-\ref{Shells_n_con2}) hold, then
\begin{equation}\label{E_R12_E_ineq}
\tilde{\bfE}_{2R,R}\ge \frac{1}{n}\tilde{\bfE}\ge \frac{1}{K_1}\left(\frac{R}{R_0}\right)^{3}\tilde{\bfE}
\end{equation}
and
\begin{equation}\label{e_R12_e_ineq}
\tilde{\bfe}_{2R,R}\le \frac{K_2}{n}\tilde{\bfe}\le
K_2\left(\frac{R}{R_0}\right)^3\tilde{\bfe}\;,
\end{equation}
where
\begin{equation}
\tilde{\bfE}=\frac{1}{T}\iint
|\nabla\otimes\bfu|^2\phi_0\,d\bfx\,dt=R_0^3\,\bfE\;
\end{equation}
is the time average of the localized enstrophy on $B(\bfo,R_0)$ and
\begin{equation}
\tilde{\bfe}=\frac{1}{2}\frac{1}{T}\iint
|\bfu|^2\phi_0^{2\delta-1}\,d\bfx\,dt=R_0^3\, \bfe\;
\end{equation}
is the time average of  the localized energy on $B(\bfo,R_0)$ with
$\phi_0$ is defined by (\ref{phi0}).

Let us note that
\begin{equation}\label{tau_obs}
\tau_0=\left(\frac{\bfe}{\bfE}\right)^{1/2}=\left(\frac{\tilde{\bfe}}{\tilde{\bfE}}\right)^{1/2}\;.
\end{equation}

Utilizing (\ref{E_R12_E_ineq}), (\ref{e_R12_e_ineq}) and
(\ref{tau_obs}) in the inequality (\ref{low_R12_bd}) gives
\begin{equation}\label{Psi_R12_low}
\tilde{\Psi}_{2R,R}\ge
\frac{1}{K_1}\left(\frac{R}{R_0}\right)^3\,\nu\tilde{\bfE}
\left(1-C_0K_1K_2\frac{\tau_0^2}{R^2}\right)\;,
\end{equation}
implying the following result.

\begin{prop}
Assume that the condition (\ref{scales_con_fin}) holds for some
$0<\gamma<1$. Then, for any $R$ satisfying (\ref{inert_range}),
\begin{equation}
c_{0,\gamma}\left(\frac{R}{R_0}\right)^3\,\nu\tilde{\bfE}\le\tilde{\Psi}_{2R,R}\;
\end{equation}
holds with $c$ and $c_{0,\gamma}$ defined in
(\ref{c_con1}) and (\ref{c_con2}).
\end{prop}

Taking the time ensemble averages in the localized energy equality
(\ref{ener_eq2}) again, this time looking for an upper bound, yields
\[
\tilde{\Psi}_{2R,R}\le \nu\tilde{\bfE}_{2R,R}+\nu\frac{C_0}{R^2}\tilde{\bfe}_{2R,R}\;.
\]
If the covering $\{A(\bfx_i,2R,R)\}_{i=1,n}$ of $B(\bfo,R_0)$ is
optimal then, in addition to (\ref{e_R12_e_ineq}),
\begin{equation}\label{e_R12_ineq}
\tilde{\bfE}_{2R,R}\le \frac{K_2}{n}\tilde{\bfE}\le
K_2\left(\frac{R}{R_0}\right)^3\tilde{\bfE}\;;
\end{equation}
hence,
\[
\tilde{\Psi}_{2R,R}\le \nu K_2\left(\frac{R}{R_0}\right)^3\tilde{\bfE}
+\nu K_2\frac{C_0}{R^2}\left(\frac{R}{R_0}\right)^3\tilde{\bfe}\;.
\]

As long as $R$ is inside the inertial range delineated in
(\ref{inert_range}),
\[
\tilde{e}\le c^2\gamma^2R^2\tilde{\bfE}, \;
\]
leading to
\[
\tilde{\Psi}_{2R,R}\le \nu K_2\left(\frac{R}{R_0}\right)^3\tilde{\bfE}
+\nu K_2{C_0}\left(\frac{R}{R_0}\right)^3c^2\gamma^2\tilde{\bfE}=
c_{1,\gamma}\left(\frac{R}{R_0}\right)^3\,\nu\tilde{\bfE}\;.
\]

Collecting the bounds on $\tilde{\Psi}_{2R,R}$ we establish the
following.

\begin{thm}\label{shells_thm}
Assume that the condition (\ref{scales_con_fin}) holds for some
$0<\gamma<1$. Then, for any $R$ satisfying (\ref{inert_range}), the
ensemble average of the time-averaged modified flux through the
shells of thickness $R$, $\tilde{\Psi}_{2R,R}$,  satisfies
\begin{equation}\label{ener_casc2}
c_{0,\gamma}\left(\frac{R}{R_0}\right)^3\,\nu\tilde{\bfE}\le\tilde{\Psi}_{2R,R}\le
c_{1,\gamma} \left(\frac{R}{R_0}\right)^3\,\nu\tilde{\bfE}\;
\end{equation}
where $c$, $c_{0,\gamma}$, and $c_{1,\gamma}$ are defined in (\ref{c_con1})
and (\ref{c_con2})
and the average is computed over a time interval $[0,2T]$ with $T\ge R_0^2/\nu$
and determined by an optimal covering $\{A(\bfx_i,2R,R)\}_{i=1,n}$ of $B(\bfo,R_0)$
(i.e. satisfying (\ref{x_i_shells_con}), (\ref{shells_n_con1}), and (\ref{Shells_n_con2})).
\end{thm}

Note that if
\[
\Psi_{2R,R}=\frac{1}{R^3}\tilde{\Psi}_{2R,R}
\]
denotes the ensemble average of the time-\emph{space} averaged
modified flux through the shells of thickness $R$ then, dividing
(\ref{ener_casc2}) by $R^{3}$, we obtain the following.

\begin{cor}
Under the conditions of the previous theorem,
\begin{equation}\label{ener_casc3}
c_{0,\gamma}\nu \bfE\le\Psi_{2R,R}\le c_{1,\gamma} \nu \bfE\;.
\end{equation}
\end{cor}

Theorem \ref{shells_thm} allows us to show locality of the
time-averaged modified flux under the assumption
(\ref{scales_con_fin}).  Indeed, the ensemble average of the
time-averaged flux through the spheres of radius $R$ satisfying
(\ref{inert_range}) is
\[
\tilde{\Psi}_R=R^3\Psi_R\;.
\]
According to Theorem \ref{balls_thm},
\[
c_{0,\gamma}\left(\frac{R}{R_0}\right)^3\,\nu\tilde{\bfE}\le\tilde{\Psi}_R\le
c_{1,\gamma} \left(\frac{R}{R_0}\right)^3\,\nu\tilde{\bfE}\;.
\]
On the other hand, the time ensemble average of the flux through the
shells between spheres of radii $R_2$ and $2R_2$, according to
Theorem \ref{shells_thm} is
\[
c_{0,\gamma}\left(\frac{R_2}{R_0}\right)^3\,\nu\tilde{\bfE}\le\tilde{\Psi}_{2R_2,R_2}\le
c_{1,\gamma} \left(\frac{R_2}{R_0}\right)^3\,\nu\tilde{\bfE}\;.
\]
Consequently,
\begin{equation}\label{time_locality}
\frac{c_{0,\gamma}}{c_{1,\gamma}}\left(\frac{R_2}{R}\right)^3\le
\frac{\tilde{\Psi}_{2R_2,R_2}}{\tilde{\Psi}_R}\le\frac{c_{1,\gamma}}{c_{0,\gamma}}
\left(\frac{R_2}{R}\right)^3\;.
\end{equation}

Thus, under the assumption (\ref{scales_con_fin}), throughout the
inertial range given by (\ref{inert_range}), the contribution of the
shells at scales comparable to $R$ is  comparable to the total flux
at scales $R$, the contribution of the the shells at scales $R_2$
much smaller than $R$ becomes negligible (ultraviolet locality) and
the flux through the shells at scales $R_2$ much bigger than $R$
becomes substantially bigger and thus essentially uncorrelated to
the flux at scales $R$ (infrared locality).

Moreover, if we choose $R_2=2^kR$ with $k$ an integer, the relation (\ref{time_locality})
becomes
\begin{equation}\label{exp_time_locality}
\frac{c_{0,\gamma}}{c_{1,\gamma}}2^{3k}\le
\frac{\tilde{\Psi}_{2^{k+1}R,2^k{R}}}{\tilde{\Psi}_R}\le\frac{c_{1,\gamma}}{c_{0,\gamma}}
2^{3k}
\end{equation}
which implies that the aforementioned manifestations of locality propagate
\emph{exponentially} in the shell number $k$.

In contrast to (\ref{time_locality}), since $\tilde{\bfe}=R_0^3\bfe$, $\tilde{\bfE}=R_0^3\bfE$,
$\tilde{\Psi}_{2R_2,R_2}=R_2^3{\Psi}_{2R_2,R_2}$ and $\tilde{\Psi}_R=R^3{\Psi}_R$,
\begin{equation}\label{space_time_locality}
\frac{c_{0,\gamma}}{c_{1,\gamma}}\le
\frac{{\Psi}_{2R_2,R_2}}{{\Psi}_R}\le\frac{c_{1,\gamma}}{c_{0,\gamma}}\;,
\end{equation}
i.e., the ensemble averages of the time-\emph{space} averaged
modified fluxes of the flows satisfying (\ref{scales_con_fin}) are
comparable throughout the scales involved in the inertial range
(\ref{inert_range}) which is consistent with the universality of the
cascade.

We conclude this section by noticing that the remarks similar to those at the
end of section \ref{balls}
can be applied here. Namely we have the following.

\begin{obs}\label{Phi_low_bds3}
{\em
Since $\tilde{\Phi}_{2R,R}^{\infty}\ge 0$, then all the {\em lower}
bounds on modified fluxes hold for the usual fluxes
$\tilde\Phi_{2R,R}$ and $\Phi_{2R,R}$; in particular, we have
\begin{equation}
\tilde{\Phi}_{2R,R}\ge c_1\left(\frac{R}{R_0}\right)^3\nu \tilde{\bfE}\,\left(1-c_2\frac{\tau_0^2}{R^2}\right)
\end{equation}
and, provided (\ref{scales_con_fin}) and (\ref{inert_range}) hold,
\begin{equation}
\tilde{\Phi}_{2R,R}\ge c_{0,\gamma}\left(\frac{R}{R_0}\right)^3\nu \tilde{\bfE}\;.
\end{equation}

Also, if a solution $\bfu$ is such that (\ref{loc_ene_ineq}) holds
with equality and in particular, if the solution is {\em regular},
then \emph{all} the estimates, including (\ref{ener_casc2}) and
(\ref{ener_casc3}), hold for the usual averaged fluxes
$\tilde{\Phi}_{2R,R}$ and $\Phi_{2R,R}=\tilde{\Phi}_{2R,R}/R^3$. }
\end{obs}

\begin{obs}\label{E'_rem3}{\em
If the additional assumptions (\ref{psi_shells_def_add1}) and (\ref{psi_shells_def_add2})
for the test functions on the shells $A(\bfx_i,2R,R)$ which are not contained
entirely in $B(\bfo,R_0)$ are not imposed, then the lower bounds in
(\ref{Psi_R12_low}) and (\ref{ener_casc2}) hold with $\tilde{\bfE}$ replaced
by the time average of  the {\em{non-localized}} in space enstrophy on $B(\bfo, R_0)$,
\[
\tilde{E}'=\frac{1}{T}\int\limits_0^{2T}
\int\limits_{B(\bfo,R_0)}|\nabla\otimes\bfu|^2\eta\,d\bfx\,dt=R_0^3E'\;.
\]
This is the case because the estimate (\ref{E_R12_E_ineq}) gets replaced with
\[\tilde{\bfE}_{2R,R}\ge\frac{1}{K_1}\left(\frac{R}{R_0}\right)^3\tilde{E}'\;.\]

Also, the estimates (\ref{time_locality}) and
(\ref{space_time_locality}) will contain the terms
$E'/\bfE(=\tilde{E}'/\tilde{\bfE})$ in the lower and $\bfE/E'$ in the upper
bounds. }\end{obs}

\bibliographystyle{plain}

\bibliography{dg1bib}

\end{document}